\documentclass[12pt]{elsarticle}
\usepackage{lineno,hyperref}
\modulolinenumbers[5]

\usepackage{amsmath}
\usepackage{amsthm}
\usepackage{amssymb}
\usepackage{amsfonts}
\usepackage{fullpage}

\newtheorem{theorem}{Theorem}[section]

\newtheorem{corollary}[theorem]{Corollary}
\newtheorem{lemma}[theorem]{Lemma}

\newtheorem{proposition}[theorem]{Proposition}

\numberwithin{equation}{section}

\def\C{{\mathbb C}}

\def\N{{\mathbb N}}

\def\R{{\mathbb R}}

\def\E{{\mathbb E}}

\def\eps{\varepsilon}

\newcommand{\abs}[1]{\left\lvert#1\right\rvert}

\begin{document}

\begin{frontmatter}
\title{Expected number of real zeros for random Freud orthogonal polynomials}

\author[myaddress]{Igor E. Pritsker\corref{mycorrespondingauthor}}
\cortext[mycorrespondingauthor]{Corresponding author}
\ead{igor@math.okstate.edu}

\author[myaddress]{Xiaoju Xie}
\ead{sophia.xie@okstate.edu}

\address[myaddress]{Department of Mathematics, Oklahoma State University, Stillwater, OK 74078, USA}


\begin{abstract}
We study the expected number of real zeros for random linear combinations of orthogonal polynomials. It is well known that Kac polynomials, spanned by monomials with i.i.d. Gaussian coefficients, have only  $(2/\pi + o(1))\log{n}$ expected real zeros in terms of the degree $n$. On the other hand, if the basis is given by orthonormal polynomials associated to a finite Borel measure with compact support on the real line, then random linear combinations have $n/\sqrt{3} + o(n)$ expected real zeros under mild conditions. We prove that the latter asymptotic relation holds for all random orthogonal polynomials on the real line associated with Freud weights, and give local results on the expected number of real zeros. We also show that the counting measures of properly scaled zeros of random Freud polynomials converge weakly to the Ullman distribution.
\end{abstract}


\end{frontmatter}

\section{Background}

Problems on the number of real zeros for polynomials with random coefficients date back to 1930s, and they are considered as some of the most classical in the area of random polynomials. These original contributions dealt with the expected number of real zeros $\E[N_n(\R)]$ for polynomials of the form $P_n(x)=\sum_{k=0}^{n} c_k x^k,$ where $\{c_k\}_{k=0}^n$ are independent and identically distributed random variables. Apparently the first paper that initiated the study is due to Bloch and P\'olya \cite{BP}, who gave an upper bound $\E[N_n(\R)] = O(\sqrt{n})$ for polynomials with coefficients selected from the set $\{-1,0,1\}$ with equal probabilities. Further results generalizing and improving that estimate were obtained by Littlewood and Offord \cite{LO1}-\cite{LO2}, Erd\H{o}s and Offord \cite{EO} and others. In particular, Kac \cite{Ka1} established the important asymptotic result
\[
\E[N_n(\R)] = (2/\pi + o(1))\log{n}\quad\mbox{as }n\to\infty,
\]
for polynomials with independent real Gaussian coefficients. More precise forms of this asymptotic were obtained by many authors, including Kac \cite{Ka2}, Wang \cite{Wa}, Edelman and Kostlan \cite{EK}. It appears that the sharpest known version is given by the asymptotic series of Wilkins \cite{Wi1}. Many additional references and further directions of work on the expected number of real zeros may be found in the books of Bharucha-Reid and Sambandham \cite{BRS}, and of Farahmand \cite{Fa}. In fact, Kac \cite{Ka1}-\cite{Ka2} found the exact formula for $\E[N_n(\R)]$ in the case of standard real independent Gaussian coefficients:
\[
\E[N_n(\R)]=\frac{4}{\pi}\displaystyle \int_0^1 \frac{\sqrt{A(x)C(x)-B^2(x)}}{A(x)} \, dx,
\]
where
\[
A(x)=\sum_{j=0}^n x^{2j},\quad
B(x)=\sum_{j=1}^n j x^{2j-1}\quad\mbox{and}\quad
C(x)=\sum_{j=1}^n j^2 x^{2j-2}.
\]
In the subsequent paper Kac \cite{Ka3}, the asymptotic result for the number of real zeros was extended to the case of uniformly distributed coefficients on $[-1, 1]$. Erd\H{o}s and Offord \cite{EO} generalized the Kac asymptotic to Bernoulli distribution (uniform on $\{-1,1\}$), while Stevens \cite{St} considered a wide class of distributions. Finally, Ibragimov and Maslova \cite{IM1,IM2} extended the result to all mean-zero distributions in the domain of attraction of the normal law.

We state a result on the number of real zeros for the random linear combinations of rather general functions. It originated in the papers of Kac \cite{Ka1}-\cite{Ka3}, who used the monomial basis, and was extended to trigonometric polynomials and other bases, see Farahmand \cite{Fa} and Das \cite{Da1}-\cite{Da2}. We are particularly interested in the bases of orthonormal polynomials, which is the case considered by Das \cite{Da1}. For any set $E\subset\C$, we use the notation $N_n(E)$ for the number of zeros of random functions \eqref{1.1} (or random orthogonal polynomials of degree at most $n$) located in $E$. The expected number of zeros in $E$ is denoted by $\E[N_n(E)],$ with $\E[N_n(a,b)]$ being the expected number of zeros in $(a,b)\subset\R.$

\begin{proposition} \label{prop1.1}
Let $[a,b]\subset\R$, and consider real valued functions $g_{j}(x)\in C^1([a,b]), \ j=0,\ldots,n,$ with $g_{0}(x)$ being a nonzero constant. Define the random function
\begin{equation} \label{1.1}
G_n(x)=\sum_{j=0}^{n} c_{j}g_{j}(x),
\end{equation}
where the coefficients $c_j$ are i.i.d. random variables with Gaussian distribution $\mathcal{N}(0, \sigma^2), \sigma > 0$. If there is $M\in\N$ such that $G_n'(x)$ has at most $M$ zeros in $(a,b)$ for all choices of coefficients, then the expected number of real zeros of $G_n(x)$ in the interval $(a,b)$ is given by
\begin{equation} \label{1.2}
\E[N_n(a,b)]=\frac{1}{\pi} \int_a^b \frac{\sqrt{A(x)C(x)-B^2(x)}}{A(x)} \, dx,
\end{equation}
where
\begin{align} \label{1.3}
A(x)=\sum_{j=0}^n g_j^{2}(x), \quad B(x)=\sum_{j=1}^n g_j(x)g_j'(x) \quad \mbox{and}\quad C(x)=\sum_{j=1}^n [g_j'(x)]^2.
\end{align}
\end{proposition}

Clearly, the original formula of Kac follows from this proposition for $g_j(x)=x^j,\ j=0,1,\ldots,n.$ For a sketch of the proof of Proposition \ref{prop1.1}, see \cite{LPX}. We note that multiple zeros are counted only once by the standard convention in all of the above results on real zeros. However, the probability of having a multiple zero for a polynomial with Gaussian coefficients is equal to 0, so that we have the same result on the expected number of zeros regardless whether they are counted with or without multiplicities.

\section{Random orthogonal polynomials}
In this paper, we consider the Freud weights
\[
W(x)=e^{-c\abs{x}^\lambda}, \, x \in \R,
\]
where $c>0$ and $\lambda>1$ are constants \cite{Fr}. For $n\geq 0$, let
\begin{equation*}
p_{n}\left( x \right) =p_{n}\left( W^2, x \right) =\gamma _{n}x^{n}+...
\end{equation*}%
denote the $n$th orthonormal polynomial with $\gamma_n>0,$ so that
\begin{equation*}
\int p_{n}p_{m}W^2 =\delta _{mn}.
\end{equation*}%
Using the orthonormal polynomials $\{p_j\}_{j=0}^{\infty}$ as the basis, we consider the ensemble of random polynomials of the form
\begin{equation} \label{2.1}
P_n(x)=\sum_{j=0}^{n}c_{j}p_{j}(x),\quad n\in\N,
\end{equation}
where the coefficients $c_0,c_1,\ldots,c_n$ are i.i.d. random variables. Such a family is often called random orthogonal polynomials. If the coefficients have Gaussian distribution, one can apply Proposition \ref{prop1.1} to study the expected number of real zeros of random orthogonal polynomials. In particular, Das \cite{Da1} considered random Legendre polynomials, and found that $\E[N_n(-1,1)]$ is asymptotically equal to $n/\sqrt{3}$. Wilkins \cite{Wi2} improved the error term in this asymptotic relation by showing that $\E[N_n(-1,1)] = n/\sqrt{3} + o(n^\eps)$ for any $\eps>0.$ Related results were obtained by Farahmand \cite{Fa1,Fa2}. For random Jacobi polynomials, Das and Bhatt \cite{DB} concluded that $\E[N_n(-1,1)]$ is asymptotically equal to $n/\sqrt{3}$ too. They also stated estimates for the expected number of real zeros of random Hermite and Laguerre polynomials, but those arguments contain significant gaps. The results of this paper provide detailed information on the expected number of real zeros for random polynomials spanned by the Freud orthogonal polynomials. In particular, they cover the case of random Hermite polynomials. Lubinsky and the authors also recently showed \cite{LPX} that if the basis is given by orthonormal polynomials associated with a finite Borel measure compactly supported on the real line, then random linear combinations have $n/\sqrt{3} + o(n)$ expected real zeros under some mild conditions on the weight. Interesting computations and pictures of zero distributions of random orthogonal polynomials may be found on the \textsc{chebfun} web page of Trefethen \cite{Tr}.

For the orthonormal polynomials $\{p_j(x)\}_{j=0}^{\infty}$, define the reproducing kernel by
\[
K_n(x,y)=\sum_{j=0}^{n-1}p_j(x)p_j(y),
\]
and the differentiated kernels by
\[
K_n^{(k,l)}(x,y)=\sum_{j=0}^{n-1}p_j^{(k)}(x)p_j^{(l)}(y),\quad k,l\in\N\cup\{0\}.
\]
The strategy is to apply Proposition \ref{prop1.1} with $g_j=p_j$, so that
\begin{align} \label{2.2}
A(x) = K_{n+1}(x,x), \quad B(x) = K_{n+1}^{(0,1)}(x,x) \quad \mbox{and}\quad C(x) = K_{n+1}^{(1,1)}(x,x).
\end{align}
We use universality limits for the reproducing kernels of orthogonal polynomials (see Levin and Lubinsky \cite{LL2}-\cite{LL3}), and asymptotic results on zeros of random polynomials (cf. Pritsker \cite{Pr}) to give asymptotics for the expected number of real zeros for a class of random orthogonal polynomials associated with Freud weights.

\begin{theorem} \label{thm2.1}
Let $W(x)=e^{-c\abs{x}^\lambda}$ be a Freud weight on $\R$, where $c>0$ and $\lambda>1$ are constants. Then the expected number of real zeros of random orthogonal polynomials \eqref{2.1} with Gaussian coefficients satisfy
\begin{align} \label{2.3}
\lim_{n\to\infty} \frac{1}{n} \E[N_n(\R)]= \frac{1}{\sqrt{3}}.
\end{align}
\end{theorem}

We note that asymptotic relation \eqref{2.3} is new even in the classical case of Hermite weight $W(x)=e^{-\frac{1}{2}x^2}$. Theorem \ref{thm2.1} is a combination of two results on zeros of random orthogonal polynomials given below. Define the constant
\begin{align} \label{2.4}
\gamma_\lambda=\displaystyle \frac{ \Gamma(\frac{1}{2})\Gamma(\frac{\lambda}{2}) }{2\Gamma(\frac{\lambda+1}{2})},
\end{align}
and the contracted version of $P_n$:
\begin{align} \label{2.5}
P_n^*(s):=P_n(a_ns),\quad n\in\N,
\end{align}
where $a_n=\gamma_\lambda^{\frac{1}{\lambda}}c^{-\frac{1}{\lambda}}n^{\frac{1}{\lambda}}$ is a positive number.

For any set $E\subset\C$, we use the notation $N_n^*(E)$ for the number of zeros of random functions $P_n^*(s)$ located in $E$. The expected number of zeros of $P_n^*(s)$ in $E$ is denoted by $\E[N_n^*(E)],$ with $\E[N_n^*\left([a,b]\right)]$ being the expected number of zeros in $[a,b]\subset\R.$

\begin{theorem} \label{thm2.2}
Let $W(x)=e^{-c\abs{x}^\lambda}$ be a Freud weight on $\R$, where $c>0$ and $\lambda>1$ are constants. If $[a,b] \subset (-1,1)$ is any closed interval, then
\begin{equation} \label{2.6}
\lim_{n\rightarrow\infty} \frac{1}{n}\E\left[N_n^*\left([a,b]\right)\right] = \frac{1}{\sqrt{3}} \mu_{w}([a,b]),
\end{equation}
where the measure
$\mu_{w}$ is given by
\[
d\mu_{w}(s)=\left( \frac{\lambda}{\pi}\int_{\abs{s}}^1 \frac{y^{\lambda-1}}{\sqrt{y^2-s^2}}\, dy  \right)ds,\  s\in [-1, 1].
\]
\end{theorem}

Note that $\mu_{w}$ is the weighted equilibrium measure for the weight $w(x)=e^{-\gamma_\lambda \abs{x}^\lambda}$ on $\R$, see \cite{ST} and the next section for details. This measure is often called the Ullman distribution.

Define the normalized zero counting measure $\tau_n=\frac{1}{n}\sum_{k=1}^n \delta_{z_k}$ for the scaled  polynomial $P_n^*(s)$ of \eqref{2.5}, where $\{z_k\}_{k=1}^n$ are its zeros, and $\delta_z$ denotes the unit point mass at $z$. We can determine the weak limit of $\tau_n$ for   random polynomials with quite general random coefficients $\{c_k\}_{k=0}^\infty.$

\begin{theorem} \label{thm2.3}
If the coefficients $\{c_k\}_{k=0}^\infty$ of random orthogonal polynomials \eqref{2.1} are complex i.i.d. random variables such that $\E[|\log|c_0||]<\infty$, then the normalized zero counting measures $\tau_n$ for the scaled  polynomials $P_n^*(s)$ converge weakly to $\mu_w$ with probability one.
\end{theorem}
Closely related results on the asymptotic zeros distribution of random orthogonal polynomials with varying weights were proved by Bloom \cite{Bl} and Bloom and Levenberg \cite{BL}, but they are not directly applicable to our case because of different normalization. Theorem \ref{thm2.3} allows to find asymptotics for the expected number of zeros in various sets. In particular, we need the following corollary for the proof of Theorem \ref{thm2.1}.

\begin{corollary} \label{cor2.4}
Suppose that the coefficients $\{c_k\}_{k=0}^\infty$ of random orthogonal polynomials \eqref{2.1} are complex i.i.d. random variables such that $\E[|\log|c_0||]<\infty$. If $E\subset\C$ is any compact set satisfying $\mu_{w}(\partial E)=0,$ then
\begin{equation} \label{2.7}
\lim_{n\rightarrow\infty} \frac{1}{n}\E\left[N_n^*(E)\right] = \mu_{w}(E),
\end{equation}
where $N_n^*(E)$ is the number of real zeros of $P_n^*(s)$ in $E$.
\end{corollary}

It is of interest to develop similar results for a more general class of weights $W$ on the whole real line, defining orthogonal polynomials $p_k$. Another direction of further work is related to relaxing conditions on random coefficients $c_k$, e.g., by considering probability distributions from the domain of attraction of normal law as in \cite{IM1,IM2}.

\section{Proofs}

Our proofs require detailed knowledge of potential theory with external fields generated by Freud weights, see \cite{LL1} and \cite{ST}.

Let $W$ be a continuous nonnegative weight function on $\R$ such that $W$ is not identically zero and $\lim_{|x|\to\infty} |x|\,W(x) = 0.$ Set $Q(x):= - \log W(x).$ The weighted equilibrium measure $\mu _W$ of $\R$  is the unique probability measure with compact support $S_W=\text{ supp } \mu_W \subset \R$ that minimizes the energy functional
\begin{equation*}
I[\mu] = - \iint \log |z-t|\, d\mu(t) d\mu(z)+2\int Q\, d\mu
\end{equation*}
amongst all probability measures $\mu$ with support on $\R$. It satisfies
\[
\int \log \frac{1}{|z-t|}\, d\mu_W(t) +Q(z)=C,\quad z \in S_W,
\]
and
\[
\int \log \frac{1}{|z-t|}\, d\mu_W(t) +Q(z)\geq C,\quad z \in \R,
\]
where $C$ is a constant. We consider the general Freud weights $W(x)=e^{-c\abs{x}^\lambda}$, where $c>0$ and $\lambda>1$ are constants, and the normalized weight $w(x)=e^{-\gamma_\lambda \abs{x}^\lambda}$, where $\gamma_\lambda$ is defined by \eqref{2.4}. In the latter case, the weighted equilibrium measure $\mu_{w}$ is given by
\[
d\mu_{w}(s)=\left( \frac{\lambda}{\pi}\int_{\abs{s}}^1 \frac{y^{\lambda-1}}{\sqrt{y^2-s^2}}\, dy  \right)ds,\  s\in [-1, 1],
\]
by Theorem 5.1 of \cite[p. 240]{ST}.

For a weight function $W(x)=e^{-Q(x)}$, the Mhaskar-Rakhmanov-Saff numbers
\[
a_{-n}<0<a_n
\]
are defined for $n \geq 1$ by the relations
\[
n=\frac{1}{\pi}\int_{a_{-n}}^{a_n}\frac{xQ'(x)}{\sqrt{(x-a_{-n})(a_n-x)}} \, dx
\]
and
\[
0=\frac{1}{\pi}\int_{a_{-n}}^{a_n}\frac{Q'(x)}{\sqrt{(x-a_{-n})(a_n-x)}} \, dx.
\]
Since $Q(x)=c\abs{x}^\lambda$ is even, we have $a_{-n}=-a_n$. Existence and uniqueness of these numbers is established in the monographs \cite{LL1}, \cite{Mh}, \cite{ST}, but goes back to earlier work of Mhaskar, Saff, and Rakhmanov. One illustration of their role is the Mhaskar-Saff identity:
\[
||PW||_{L_\infty (\R)}=||PW||_{L_\infty ([-a_n, a_n])},
\]
which is valid for all polynomials $P$ of degree at most $n$. It is known that the Mhaskar-Rakhmanov-Saff number associated with the Freud weight $W(x)=e^{-c\abs{x}^\lambda}$ is given by
\[
a_n=\gamma_\lambda^{\frac{1}{\lambda}}c^{-\frac{1}{\lambda}}n^{\frac{1}{\lambda}}.
\]
See page 308 of \cite{ST} for further details. We define the Mhaskar-Rakhmanov-Saff interval $\Delta_n$ as $\Delta_n :=[-a_n, a_n]$. The linear transformation
\[
L_n(x)=\frac{x}{a_n},\ x\in \R,
\]
maps $\Delta_n$ onto $[-1, 1]$. Its inverse is
\[
L_n^{[-1]}(s)=a_ns,\ s\in \R.
\]
For $\varepsilon\in(0,1)$, we let
\[
J_n(\varepsilon)=L_n^{[-1]}[-1+\varepsilon, 1-\varepsilon]=(1-\varepsilon)[-a_n, a_n].
\]
Then the equilibrium density is defined as
\[
\sigma_n(x)=\displaystyle \frac{\sqrt{(a_n+x)(a_n-x)}}{\pi^2}\int_{-a_n}^{a_n} \frac{Q'(s)-Q'(x)}{s-x}\frac{ds}{\sqrt{(a_n+s)(a_n-s)}}, \, x\in \Delta_n.
\]
The equilibrium density satisfies \cite[p. 41]{LL1}:
\[
\displaystyle \int_{-a_n}^{a_n} \log \frac{1}{\abs{x-s}}\sigma_n(s)\, ds + Q(x)=C, \, x\in \Delta_n,
\]
and
\[
\displaystyle \int_{-a_n}^{a_n} \log \frac{1}{\abs{x-s}}\sigma_n(s)\, ds + Q(x)\geq C, \, x\in \R.
\]
Note that the measure $\sigma_n(x)dx$ has total mass $n$:
\[
\displaystyle \int_{-a_n}^{a_n} \sigma_n(x)\, dx=n.
\]
We also define the normalized version of $\sigma_n$ as follows:
\begin{align*}
\sigma_n^*(s) := \frac{a_n}{n}\sigma_n(a_ns),\quad s\in [-1,1].
\end{align*}
Note that
\[
\int_{-1}^1\sigma_n^*(s) \, ds=1.
\]
For details on $\sigma_n$, one should consult the book \cite{LL1} by Levin and Lubinsky.

\begin{lemma} \label{lem3.1}
For a Freud weight $W(x)=e^{-c\abs{x}^\lambda}$, where $c>0$ and $\lambda>1$ are constants, the normalized equilibrium density satisfies
\[
\sigma_{n}^*(s)\, ds = d\mu_{w}(s) \text{ for all } n \in \N,\quad s\in [-1, 1].
\]
That is,
\[
\sigma_{n}^*(s)=\frac{\lambda}{\pi}\int_{\abs{s}}^1 \frac{y^{\lambda-1}}{\sqrt{y^2-s^2}}\, dy \text{ for all } n\in \N,\quad s\in [-1, 1].
\]
\end{lemma}
\begin{proof}
Recall that $\sigma_n$ satisfies
\[
\displaystyle \int_{-a_n}^{a_n} \log \frac{1}{\abs{x-y}}\sigma_n(y)\, dy + c\abs{x}^\lambda=C_1, \quad x\in [-a_n, a_n],
\]
and
\[
\displaystyle \int_{-a_n}^{a_n} \log \frac{1}{\abs{x-y}}\sigma_n(y)\, dy + c\abs{x}^\lambda \geq C_1, \quad x\in \R,
\]
where $a_n=\gamma_\lambda^{\frac{1}{\lambda}}c^{-\frac{1}{\lambda}}n^{\frac{1}{\lambda}}$.
The changes of variables $x=a_n s$ and $y=a_n t$ reduce the above relations to
\[
\displaystyle \int_{-1}^{1} \log \frac{1}{\abs{s-t}}\sigma_n^*(t)\, dt + \gamma_{\lambda}\abs{s}^\lambda=C_2, \ s\in [-1, 1],
\]
and
\[
\displaystyle \int_{-1}^{1} \log \frac{1}{\abs{s-t}}\sigma_n^*(t)\, dt + \gamma_{\lambda}\abs{s}^\lambda\geq C_2,\ s\in \R.
\]
Invoking Theorem 3.1 of \cite[p. 43]{ST}, we deduce that
\[
\sigma_{n}^*(s)\, ds = d\mu_{w}(s) \text{ for all } n \in \N,\ s\in [-1, 1].
\]
On the other hand, recalling that
\[
d\mu_{w}(s)=\left( \frac{\lambda}{\pi}\int_{\abs{s}}^1 \frac{y^{\lambda-1}}{\sqrt{y^2-s^2}}\, dy  \right)ds,\  s\in [-1, 1],
\]
we obtain that
\[
\sigma_{n}^*(s)=\frac{\lambda}{\pi}\int_{\abs{s}}^1 \frac{y^{\lambda-1}}{\sqrt{y^2-s^2}}\, dy \text{ for all } n\in \N,\ s\in [-1, 1].
\]
\end{proof}

\begin{proof}[Proof of Theorem \ref{thm2.2}]
In this case, $W(x)=e^{-c\abs{x}^\lambda}, x\in \R$, where $c>0$ and $\lambda>1$ are constants. The strategy is to apply Theorem 1.6 of \cite{LL2}. It states that for all $r,s\geq 0$ and any $\varepsilon\in(0,1)$, we have
\begin{equation} \label{3.7}
\frac{W^2(x)K_n^{(r,s)}(x,x)}{(\sigma_n(x))^{r+s+1}}=\sum_{j=0}^r \left( \begin{array}{c} r \\ j \end{array}  \right) \sum_{k=0}^s \left( \begin{array}{c} s \\ k \end{array}  \right)\tau_{j,k}\pi^{j+k}\left(\frac{Q'(x)}{\sigma_n(x)}\right)^{r+s-j-k}+o(1)\quad\mbox{as } n\to\infty,
\end{equation}
uniformly for $x \in J_n(\varepsilon)$, where
\begin{equation*}
\tau _{j,k}=\left\{
\begin{tabular}{ll}
$0,$ & $j+k$ odd, \\
$(-1)^{(j-k)/2}\frac{1}{j+k+1},$ & $j+k$ even.
\end{tabular}
\right.
\end{equation*}
That is, uniformly in $x \in J_{n+1}(\varepsilon)$,
\[
\frac{W^2(x)K_{n+1}^{\left(0,0\right)}\left(x,x\right)}{\sigma_{n+1}(x)}=1+o(1)\quad\mbox{as } n\to\infty,
\]
\[
\frac{W^2(x)K_{n+1}^{\left(0,1\right)}\left(x,x\right)}{(\sigma_{n+1}(x))^{2}}=\frac{Q'(x)}{\sigma_{n+1}(x)}+o(1)\quad\mbox{as } n\to\infty,
\]
and
\[
\frac{W^2(x)K_{n+1}^{\left(1,1\right)}\left(x,x\right)}{(\sigma_{n+1}(x))^{3}}=\left(\frac{Q'(x)}{\sigma_{n+1}(x)}\right)^2+\frac{\pi^2}{3}+o(1)\quad\mbox{as } n\to\infty.
\]
Applying \eqref{1.2} with \eqref{2.2}, we obtain that
\begin{equation} \label{3.8}
\frac{1}{n} \E\left[N_{n}\left([l, q]\right)\right] = \frac{1}{\pi n}\int_l^q \sqrt{\frac{K_{n+1}^{\left( 1,1\right) }\left( x,x\right) }{K_{n+1}^{\left( 0,0\right) }\left( x,x\right)}-\left( \frac{K_{n+1}^{\left(0,1\right) }\left(x,x\right)}{K_{n+1}^{\left( 0,0\right) }\left( x,x\right) }\right) ^{2}} dx
\end{equation}
for any closed interval $[l, q] \subset J_{n+1}(\varepsilon)$ ($l, q$ may depend on $n$). Now, uniformly for $x \in J_{n+1}(\varepsilon)$,
\[
\frac{K_{n+1}^{\left( 1,1\right) }\left( x,x\right) }{K_{n+1}^{\left( 0,0\right) }\left( x,x\right)}= \sigma_{n+1}^2(x) \left( \left(\frac{Q'(x)}{\sigma_{n+1}(x)}\right)^2+\frac{\pi^2}{3}+o(1)\right)(1+o(1))^{-1}
\]
and
\begin{align*}
\left( \frac{K_{n+1}^{\left(0,1\right) }\left(x,x\right)}{K_{n+1}^{\left( 0,0\right) }\left( x,x\right) }\right) ^{2}&=\sigma_{n+1}^2(x)\left(\frac{Q'(x)}{\sigma_{n+1}(x)}+o(1)\right)^2(1+o(1))^{-2}\\
&=\sigma_{n+1}^2(x)\left( \left(\frac{Q'(x)}{\sigma_{n+1}(x)}\right)^2 + \frac{Q'(x)}{\sigma_{n+1}(x)}o(1)+o(1)\right)(1+o(1))^{-1}.
\end{align*}
We thus obtain that
\begin{align*}
&\frac{1}{n}\E\left[ N_{n}([l, q])\right]  \\
&=\frac{1}{\pi } \int_l^q \frac{1}{n}\sqrt{\sigma_{n+1}^2(x) \frac{\left(\frac{Q'(x)}{\sigma_{n+1}(x)}\right)^2+\frac{\pi^2}{3}+o(1)}{1+o(1)}- \sigma_{n+1}^2(x)\frac{\left(\frac{Q'(x)}{\sigma_{n+1}(x)}\right)^2+\frac{Q'(x)}{\sigma_{n+1}(x)}o(1)+o(1)}{1+o(1)} }\,dx \\
&=\frac{1}{\pi } \int_l^q \frac{\sigma_{n+1}(x)}{n} \sqrt{\frac{\pi^2}{3}+\left(\frac{Q'(x)}{\sigma_{n+1}(x)}\right)^2 o(1) + \frac{Q'(x)}{\sigma_{n+1}(x)}o(1)+o(1)}\, dx \quad\mbox{as } n\to\infty.
\end{align*}
Note that the number $N_{n}(E)$ of real zeros of $P_n(x)$ in $E$ equals the number $N_{n}^*(E^*)$ of real zeros of $P_n^*(s)$ in $E^*:=E/a_{n+1}$, since $L_{n+1}$ is a bijection. Since $[a,b] \subset (-1, 1)$ is a closed interval, we have that $[a_{n+1}a,a_{n+1}b] \subset J_{n+1}(\varepsilon)=(1-\varepsilon)[-a_{n+1},a_{n+1}]$ provided $\max\{\abs{a}, \abs{b}\}\leq 1-\varepsilon$ for some constant $\varepsilon \in (0, 1)$. Hence
\begin{align} \label{3.9}
\frac{1}{n}\E\left[ N_{n}^*([a,b])\right]&=\frac{1}{n}\E\left[ N_{n}([a_{n+1}a,a_{n+1}b])\right] \\
&=\frac{1}{\pi} \int_{a_{n+1}a}^{a_{n+1}b} \frac{\sigma_{n+1}(x)}{n} \sqrt{  \frac{\pi^2}{3}+\left(\frac{Q'(x)}{\sigma_{n+1}(x)}\right)^2 o(1) + \frac{Q'(x)}{\sigma_{n+1}(x)}o(1)+o(1)}\, dx \nonumber\\
&=\frac{n+1}{\pi n} \int_a^b \sigma_{n+1}^*(s)\, \sqrt{ \frac{\pi^2}{3} + \left(\left(\frac{Q'(a_{n+1}s)}{\sigma_{n+1}(a_{n+1}s)}\right)^2 + \frac{Q'(a_{n+1}s)}{\sigma_{n+1}(a_{n+1}s)}+1\right)o(1)}\, ds. \nonumber
\end{align}
We show that $\abs{Q'(x)/\sigma_{n+1}(x)}\leq C$ on $x\in[a_{n+1}a, a_{n+1}b]$ for some constant $C>0$ independent of $n$. Recall that $[a_{n+1}a, a_{n+1}b] \subset J_{n+1}(\varepsilon)=(1-\varepsilon)[-a_{n+1},a_{n+1}]$, for some $\varepsilon \in (0, 1)$. It is clear that $Q'(x)/\sigma_{n+1}(x)$ is an odd function of $x \in [a_{n+1},0)\cup (0, a_{n+1}]$, so that we only need to consider the interval $(0, a_{n+1}(1-\varepsilon)]$. First note that for $x\in (0, a_{n+1})$,
\[
\frac{Q'(x)}{\sigma_{n+1}(x)}=c^{1-1/\lambda}\gamma_\lambda^{1/\lambda}\pi\frac{x^{\lambda-1}}{(n+1)^{1-1/\lambda}}\left( \int_{x/a_{n+1}}^1 \frac{y^{\lambda-1}}{\sqrt{y^2-x^2/a_{n+1}^2}}\, dy\right)^{-1}.
\]
Since $\lambda>1$, we can estimate the integral in the above formula for $x \in (0, a_{n+1}(1-\varepsilon)]$ as follows:
\[
\int_{x/a_{n+1}}^1 \frac{y^{\lambda-1}}{\sqrt{y^2-x^2/a_{n+1}^2}}\, dy \geq \int_{x/a_{n+1}}^1 \frac{y^{\lambda-1}}{\sqrt{y^2}}\, dy=\frac{1-(x/a_{n+1})^{\lambda-1}}{\lambda-1}.
\]
Using this estimate, we see that for $x \in (0, a_{n+1}(1-\varepsilon)]$,
\begin{align*}
\abs{\frac{Q'(x)}{\sigma_{n+1}(x)}} &\leq c^{1-1/\lambda}\gamma_\lambda^{1/\lambda}\pi\frac{x^{\lambda-1}}{(n+1)^{1-1/\lambda}}\frac{\lambda-1}{1-(x/a_{n+1})^{\lambda-1}}\\
&\leq c^{1-1/\lambda}\gamma_\lambda^{1/\lambda}\pi\frac{(a_{n+1}(1-\varepsilon))^{\lambda-1}}{(n+1)^{1-1/\lambda}}\frac{\lambda-1}{1-(1-\varepsilon)^{\lambda-1}}\\
&=\frac{\gamma_\lambda\pi(\lambda-1)(1-\varepsilon)^{\lambda-1}}{1-(1-\varepsilon)^{\lambda-1}}=:C.
\end{align*}
Applying \eqref{3.9} and Lemma \ref{lem3.1}, we obtain that
\begin{align*}
\frac{1}{n}\E\left[ N_{n}^*([a,b])\right]
&=\frac{1}{\pi }\left(1+\frac{1}{n}\right) \int_a^b \sigma_{n+1}^*(s)\, \sqrt{\frac{\pi^2}{3}+o(1)}\, ds\\
&=\frac{1+o(1)}{\sqrt{3} } \int_a^b  \sigma_{n+1}^*(s) \, ds\\
&=\frac{1+o(1)}{\sqrt{3} } \int_a^b d\mu_{w} (s).
\end{align*}
To complete the proof, we pass to the limit as $n\rightarrow\infty$:
\[
\lim_{n\rightarrow\infty} \frac{1}{n}\E\left[N_n^*\left([a,b]\right)\right] = \frac{1}{\sqrt{3}} \int_a^b \, d\mu_{w} (s)=\frac{1}{\sqrt{3}}\mu_{w}\left([a,b]\right).
\]
\end{proof}

\begin{proof}[Proof of Theorem \ref{thm2.3}]
Following \cite{ST}, we call a sequence of monic polynomials $\{Q_n\}_{n=1}^\infty$, with $ \deg(Q_n)=n,$ asymptotically extremal with respect to the weight $w$ if it satisfies
\[
\lim_{n\to\infty} \|w^n Q_n\|_\R^{1/n} = e^{-F_w},
\]
where $\|\cdot\|_\R$ is the supremum norm on $\R$ and $F_w=\log{2}+1/\lambda$ is the modified Robin constant corresponding to $w$, see \cite[p. 240]{ST}. Theorem 4.2 of \cite[p. 170]{ST} states that any sequence of such asymptotically extremal monic polynomials have their zeros distributed according to the measure $\mu_w$. Namely, the normalized zero counting measures of $Q_n$ converge weakly to $\mu_w$. We show that the monic polynomials
\[
Q_n^*(x) := P_n^*(x)/(c_n\gamma_n a_n^n),\quad n\in\N,
\]
are asymptotically extremal in this sense with probability one, so that the result of Theorem \ref{thm2.3} follows.

Using orthogonality, we obtain for polynomials defined in \eqref{2.1} that
\[
\int_{-\infty}^\infty |P_n(x)|^2 W^2(x)\,dx = \sum_{k=0}^n |c_k|^2.
\]
Hence
\[
\max_{0\le k\le n} |c_k| \le \left(\int_{-\infty}^\infty |P_n(x)|^2 W^2(x)\,dx \right)^{1/2} \le (n+1) \max_{0\le k\le n} |c_k|.
\]
Lemma 4.2 of \cite{Pr} (see (4.6) there) implies that
\[
\lim_{n\to\infty} \left(\int_{-\infty}^\infty |P_n(x)|^2 W^2(x)\,dx \right)^{1/(2n)} = \lim_{n\to\infty} \left( \max_{0\le k\le n} |c_k| \right)^{1/n} = 1
\]
with probability one. Applying the Nikolskii-type inequalities of Theorem 6.1 and Theorem 6.4 from \cite{MS}, we obtain that the same holds for the supremum norm:
\[
\lim_{n\to\infty} \left\|w^n P_n^* \right\|_\R^{1/n} = \lim_{n\to\infty} \left\|P_n W \right\|_\R^{1/n} = 1
\]
with probability one. Recall that the leading coefficients of orthonormal polynomials $p_n$ satisfy
\[
\lim_{n\to\infty} \gamma_n^{1/n}\,n^{1/\lambda} = 2 c^{1/\lambda} \gamma_\lambda^{-1/\lambda} e^{1/\lambda}
\]
by Theorem 1.2 of \cite[p. 362]{ST}. We also use below that $\lim_{n\to\infty} |c_n|^{1/n} = 1$ with probability one by Lemma 4.2 of \cite{Pr}. It follows that
\begin{align*}
\lim_{n\to\infty} \left\|w^n Q_n^* \right\|_\R^{1/n} &= \lim_{n\to\infty} \left\|w^n P_n^* \right\|_\R^{1/n} \lim_{n\to\infty} |c_n \gamma_n|^{-1/n}\,|a_n|^{-1} = \lim_{n\to\infty} \left(\gamma_n^{1/n} n^{1/\lambda} c^{-1/\lambda} \gamma_\lambda^{1/\lambda} \right)^{-1} \\ &= \left( 2 c^{1/\lambda} \gamma_\lambda^{-1/\lambda} e^{1/\lambda} c^{-1/\lambda} \gamma_\lambda^{1/\lambda} \right)^{-1} = e^{-(\log{2}+1/\lambda)} = e^{-F_w}.
\end{align*}

\end{proof}

\begin{proof}[Proof of Corollary \ref{cor2.4}]
Consider the normalized zero counting measure $\tau_n=\frac{1}{n}\sum_{k=1}^n \delta_{z_k}$ for the scaled  polynomial $P_n^*(s)$ of \eqref{2.4}, where $\{z_k\}_{k=1}^n$ are the zeros of that polynomial, and $\delta_z$ denotes the unit point mass at $z$. Theorem \ref{thm2.3} implies that measures $\tau_n$ converge weakly to $\mu_w$ with probability one. Since $\mu_w(\partial E)=0,$ we obtain that $\tau_n\vert_E$ converges weakly to $\mu_w\vert_E$ with probability one by Theorem $0.5^\prime$ of \cite{La} and Theorem 2.1 of \cite{Bi}. In particular, we have that the random variables $\tau_n(E)$ converge to $\mu_w(E)$ with probability one. Hence this convergence holds in $L^p$ sense by the Dominated Convergence Theorem, as $\tau_n(E)$ are uniformly bounded by 1, see Chapter 5 of \cite{Gut}. It follows that
\[
\lim_{n\to\infty} \E[|\tau_n(E) - \mu_w(E)|] = 0
\]
for any compact set $E$ such that $\mu_w(\partial E)=0,$ and
\[
\left|\E[\tau_n(E) - \mu_w(E)]\right| \le \E[|\tau_n(E) - \mu_w(E)|] \to 0 \quad\text{as } n\to\infty.
\]
But $\E[\tau_n(E)]=\E[N_n^*(E)]/n$ and $\E[\mu_w(E)]=\mu_w(E),$ which immediately gives \eqref{2.7}.
\end{proof}

\begin{proof}[Proof of Theorem \ref{thm2.1}]
Theorem \ref{thm2.2} gives that
\[
\lim_{n\rightarrow\infty} \frac{1}{n}\E\left[N_n^*\left([a,b]\right)\right] = \frac{1}{\sqrt{3}} \mu_{w}([a,b])
\]
for any interval $[a,b]\subset (-1,1)$. Note that both $\E\left[N_n^*\left(H\right)\right]$ and $\mu_{w}(H)$ are additive functions of the set $H$. Moreover, they both vanish when $H$ is a single point by \eqref{2.7} and the absolute continuity of $\mu_{w}$ with respect to Lebesgue measure on $S_{w}=[-1,1]$.
Hence \eqref{2.7} gives that
\[
\lim_{n\rightarrow\infty} \frac{1}{n}\E\left[N_n^*\left(\R\setminus (-1,1)\right)\right] = \mu_{w}(\R\setminus(-1,1)) = 0.
\]
It now follows that
\[
\lim_{n\to\infty} \frac{1}{n} \E[N_n^*(\R)] = \frac{1}{\sqrt{3}} \mu_{w}((-1,1)) = \frac{1}{\sqrt{3}}.
\]
To complete the proof, observe that $N_n^*(\R)=N_n(\R)$, so that $\E[N_n^*(\R)]=\E[N_n(\R)]$, since $L_{n+1}$ is a bijection for each fixed $n$.
Therefore (\ref{2.3}) is proved.
\end{proof}

\section*{Acknowledgements}

Research of the first author was partially supported by the National Security Agency (grant H98230-12-1-0227). Work of the second author is done towards completion of her Ph.D. degree at Oklahoma State University under the direction of the first author.

\bibliographystyle{elsarticle-num}

\begin{thebibliography}{00}
\bibitem{BRS} A. T. Bharucha-Reid and M. Sambandham, Random Polynomials, Academic Press, Orlando, 1986.

\bibitem{Bi} P. Billingsley, Convergence of Probability Measures, John Wiley \& Sons, Inc., New York, 1999.

\bibitem{BP}  A. Bloch and G. P\'olya, On the roots of certain algebraic equations, Proc. London Math. Soc. 33 (1932), 102--114.

\bibitem{Bl} T. Bloom, Random polynomials and (pluri)potential theory, Ann. Polon. Math. 91 (2007), 131--141.

\bibitem{BL} T. Bloom and N. Levenberg, Random polynomials and pluripotential-theoretic extremal functions, Potential Anal. 42 (2015), 311--334.

\bibitem{Da1} M. Das, Real zeros of a random sum of orthogonal polynomials, Proc. Amer. Math. Soc. 27 (1971), 147--153.

\bibitem{Da2} M. Das, The average number of real zeros of a random trigonometric polynomial, Proc. Camb. Phil. Soc. 64 (1968), 721--729.

\bibitem{DB} M. Das and S. S. Bhatt, Real roots of random harmonic equations, Indian J. Pure Appl. Math. 13 (1982), 411--420.

\bibitem{EK} A. Edelman and E. Kostlan, How many zeros of a random polynomial are
real?, Bull. Amer. Math. Soc. 32 (1995), 1--37.

\bibitem{EO}  P. Erd\H{o}s and A. C. Offord, On the number of real roots of a random algebraic equation, Proc. London Math. Soc. 6 (1956), 139--160.

\bibitem{Fa} K. Farahmand, Topics in Random Polynomials, Pitman Res. Notes Math. 393 (1998).

\bibitem{Fa1} K. Farahmand, Level crossings of a random orthogonal polynomial, Analysis 16 (1996), 245--253.

\bibitem{Fa2} K. Farahmand, On random orthogonal polynomials, J. Appl. Math. Stochastic Anal. 14 (2001), 265--274.

\bibitem{Fr} G. Freud, Orthogonal Polynomials, Akademiai Kiado/Pergamon Press, Budapest, 1971.

\bibitem{Gut} A. Gut, Probability: A Graduate Course, Springer, New York, 2005.

\bibitem{IM1} I. A. Ibragimov and N. B. Maslova, The average number of zeros of random polynomials, Vestnik Leningrad University 23 (1968), 171--172.

\bibitem{IM2} I. A. Ibragimov and N. B. Maslova, The mean number of real zeros of random polynomials. I. Coefficients with zero mean, Theory Probab. Appl. 16 (1971), 228--248.

\bibitem{Ka1} M. Kac, On the average number of real roots of a random algebraic equation, Bull. Amer. Math. Soc. 49 (1943), 314--320.

\bibitem{Ka2} M. Kac, On the average number of real roots of a random algebraic equation. II, Proc. London Math. Soc. 50 (1948), 390--408.

\bibitem{Ka3} M. Kac, Nature of probability reasoning, Probability and related topics in physical sciences, Proceedings of the Summer Seminar, Boulder, Colo., 1957, Vol. I Interscience Publishers, London-New York, 1959.

\bibitem{La} N. S. Landkof, Foundations of Modern Potential Theory, Springer-Verlag, New York - Heidelberg, 1972.

\bibitem{LL1} E. Levin and D. S. Lubinsky, Orthogonal Polynomials for Exponential Weights, Springer, New York, 2001.

\bibitem{LL2} E. Levin and D. S. Lubinsky, Applications of universality limits to zeros and reproducing kernels of orthogonal polynomials, J. Approx. Theory 150 (2008), 69--95.

\bibitem{LL3} E. Levin and D. S. Lubinsky, Universality limits for exponential weights, Constr. Approx. 29 (2009), 247--275.

\bibitem{LO1} J. E. Littlewood and A. C. Offord, On the number of real roots of a random algebraic equation, J. Lond. Math. Soc. 13 (1938), 288--295.

\bibitem{LO2} J. E. Littlewood and A. C. Offord, On the number of real roots of a random algebraic equation. II, Proc. Camb. Philos. Soc. 35 (1939), 133--148.

\bibitem{LPX} D. S. Lubinsky, I. E. Pritsker, and X. Xie, Expected number of real zeros for random linear combinations of orthogonal polynomials, submitted to Proc. Amer. Math. Soc. arXiv:1503.06376

\bibitem{Mh} H. N. Mhaskar, Introduction to the Theory of Weighted Polynomial Approximation, World Scientific, Singapore, 1996.

\bibitem{MS} H. N. Mhaskar and E. B. Saff, Extremal problems for polynomials with exponential weights, Trans. Amer. Math. Soc. 285 (1984), 203--234.

\bibitem{Pr} I. E. Pritsker, Zero distribution of random polynomials, J. d'Analyse Math., to appear. arXiv:1409.1631

\bibitem{ST} E. B. Saff and V. Totik, Logarithmic Potentials with External Fields, Springer, New York, 1997.

\bibitem{St} D. C. Stevens, The average number of real zeros of a random polynomial, Comm. Pure Appl. Math. 22 (1969), 457--477.

\bibitem{Tr} N. Trefethen, Roots of random polynomials on an interval, \textsc{chebfun} web page http://www.chebfun.org/examples/roots/RandomPolys.html

\bibitem{Wa} Y. J. Wang, Bounds on the average number of real roots of a random algebraic equation, Chinese Ann. Math. Ser. A. 4 (1983), 601--605.

\bibitem{Wi1} J. E. Wilkins, Jr. An asymptotic expansion for the expected number of real zeros of a random polynomial, Proc. Amer. Math. Soc. 103 (1988), 1249--1258.

\bibitem{Wi2} J. E. Wilkins, Jr. The expected value of the number of real zeros of a random sum of Legendre polynomials, Proc. Amer. Math. Soc. 125 (1997), 1531--1536.
\end{thebibliography}

\end{document}